\documentclass[english]{amsart}

\usepackage{babel}
\usepackage{amstext}
\usepackage{amsmath}
\usepackage{amsfonts}
\usepackage{latexsym}
\usepackage{ifthen}
\usepackage{xypic}
\xyoption{all}
\newcommand\sF{{\mathcal F}}

\newcommand\sI{{\mathcal I}}

\newcommand\sL{{\mathcal L}}
\newcommand\sO{{\mathcal O}}
\newcommand\sS{{\mathcal S}}
\newcommand\sM{{\mathcal M}}

\newcommand\sC{{\mathcal C}}
\newcommand\sT{{\mathcal T}}
\newcommand\bR{{\mathbb R}}

\newcommand\bC{{\mathbb C}}
\newcommand\bQ{{\mathbb Q}}

\newcommand\bP{{\mathbb P}}
\newcommand\sN{{\mathcal N}}

\newcommand\sHom{{\mathcal Hom}}

\newcounter{lemma}

\newtheorem{lemma1}[lemma]{\setcounter{equation}{0}}

\newenvironment{lemma}{\begin{lemma1}{\bf Lemma.}}{\end{lemma1}}

\newenvironment{theorem}{\begin{lemma1}{\bf Theorem.}}{\end{lemma1}}
\newenvironment{question}{\begin{lemma1}{\bf Question.}}{\end{lemma1}}

\newenvironment{proposition}{\begin{lemma1}{\bf Proposition.}}{\end{lemma1}}

\newenvironment{corollary}{\begin{lemma1}{\bf Corollary.}}{\end{lemma1}}

\newenvironment{remark}{\begin{lemma1}{\bf Remark.}\rm}{\end{lemma1}}
\newenvironment{definition}{\begin{lemma1}{\bf Definition.}}{\end{lemma1}}

\newenvironment{conjecture}{\begin{lemma1}{\bf Conjecture.}}{\end{lemma1}}

\newenvironment{Induction Step}{\begin{lemma1}{\bf Induction Step.}}{\end{lemma1}}
\newenvironment{Proof of Theorem 1.2}{\begin{lemma1}{\bf Proof of Theorem 1.2.}}{\end{lemma1}}

\newenvironment{ps-example}{\begin{lemma1}{\bf Pseudo-Example.}\rm}{\end{lemma1}}

\title {Compact subvarieties with ample normal bundles, algebraicity and cones of cycles} \author{Thomas Peternell}
\date{\today} 
\begin{document}

\maketitle

\tableofcontents 

\section{Introduction} 
In this note we study two features of submanifolds (or possibly singular subvarieties) $Z$ with ample normal bundle in a compact K\"ahler 
manifold $X$. 

\vskip .2cm 
First we ask whether $Z$ influences the algebraic dimension $a(X)$, i.e., the maximal number of algebraically independent meromorphic functions. 
One expects (for simplicity we shall assume $Z$ smooth) 

\begin{conjecture} \label{conji} Let $X$ be a compact K\"ahler manifold containing a compact submanifold $Z$ of dimension $d \geq 1$ with ample normal bundle. Then 
$a(X) \geq d+1.$ 
\end{conjecture} 

If $d  = \dim X -1,$ then it is classically known that $X$ is projective, but in higher codimension there are only a few results, \cite{BM04}, \cite{OP04}. 
These results will be explicity discussed in Section 3. We just mention here that for threefolds containing a curve with ample normal bundle, the conjecture holds up to 
a strange phenomenon concerning threefolds without meromorphic functions. 
Our results  can be summarized as follows.

\begin{theorem} Conjecture \ref{conji} has a positive answer in the following cases; actually in all cases $X$ is even projective. 
\begin{enumerate} 
\item $Z$ moves in a family covering $X;$
\item $X$ is hyperk\"ahler with $a(X) \geq 1;$
\item $Z$ is uniruled. 
\end{enumerate} 
\end{theorem} 

Up to the standard conjecture that compact K\"ahler manifolds with non-pseudo-effective canonical bundles must be uniruled, 
assertion (3) even holds if $Z$ is not of general type. 
These results might suggest that Conjecture \ref{conji} has a stronger version saying that $X$ must be projective. But this is very unlikely; we exhibit - based on \cite{OP04} - 
a candidate for a K\"ahler threefold $X$ with $a(X) = 2$ containing a curve with ample normal bundle. However a construction is still missing. 

\vskip .2cm  The second part of the paper is concerned with {\it projective} manifolds $X$ and curves $C \subset X$ with ample normal bundles. In the ``dual'' situation 
of a hypersurface $Y$ with ample normal bundle, the line bundle $\sO_X(Y)$ is big and therefore in the interior of the pseudo-effective cone. Therefore we expect that 
the class $[C]$ is in the interior of the Mori cone $\overline{NE}(X)$ of curves:

\begin{conjecture} Let $X$ be a projective manifold, $C \subset X$ a curve with ample normal bundle. Then $[C]$ is in the interior of $\overline{NE}(X).$ 
\end{conjecture} 

Equivalently, if $L$ is any nef line bundle, such that $L \cdot C = 0,$ then $L \equiv 0. $ 
We prove

\begin{theorem} \label{thm 0} Let $X$ be a projective manifold, $C \subset X$ a smooth curve with ample normal bundle and $L$ a nef line bundle on $X.$
If $H^0(X,mL) \ne 0$ for some $m > 0$ and if $L \cdot C = 0,$ then $L \equiv 0.$ 
\end{theorem} 

The key is the fact, due to M.Schneider, that the complement $X \setminus C$ is $(n-1)-$convex in the sense of Andreotti-Grauert where $n = \dim X.$ Therefore we prove more generally

\begin{theorem} Let $Z$ be an $(n-1)-$convex manifold of dimension $n$. Let $Y \subset Z$ be a compact hypersurface (possibly  reducible and non-reduced). Then 
the normal bundle $N_{Y/Z}$ cannot be nef. 
\end{theorem} 

It is tempting to ask for generalizations for $q-$convex manifolds and subvarieties of higher codimension; we discuss this in Section 4. We also prove some further
results in the spirit of Theorem \ref{thm 0}.

\section{Preliminaries} 

\noindent
We start by fixing some notations. 
\vskip .2cm \noindent
(1) Given a complex manifold $X$ and a complex subspace $Y \subset X$ with defining ideal sheaf $\sI,$
the normal sheaf $\sN_{Y/X}$ of $Y$ is given by
$$ \sN_{Y/X} = \sHom(\sI/\sI^2,\sO_Y) = (\sI/\sI^2)^*.$$
\vskip .2cm \noindent
(2) A coherent sheaf $\sS$ on a compact complex space $X$ is {\it ample} if the tautological line bundle $\sO(1)$ on $\bP(\sS)$ is 
ample. Here the projectivization is taken in Grothendieck's sense (cp. e.g. [Ha77], II.7). For details on ample sheaves we refer e.g. to
\cite{AT82}. 
\vskip .2cm \noindent (3) The algebraic dimension, the transdental degree over $\bC$ of the field of meromorphic functions,  of a compact manifold $X$ will be denoted by $a(X).$ 
\vskip .2cm \noindent (4) 
A compact K\"ahler manifold (or a manifold in class $\sC$, i.e. bimeromorphic to a K\"ahler manifold) is called {\it simple}, if there
is no proper compact subvariety through a very general point of $X.$ Equivalently, there is no family of proper subvarieties of $X$ which 
cover $X$. In particular $a(X) = 0.$ \\
The only known examples of simple manifolds are - up to bimeromorphic equivalence - general complex tori and ``general'' hyperk\"ahler manifolds. 
In dimension 3, Brunella [Br06] proved that the canonical bundle of a simple manifold $X$ must be pseudo-effective. It is also known \cite{DP03} that if
a simple threefold $X$ has a minimal model $X'$, i.e. $X'$ is a normal K\"ahler space with only say terminal singularities and $K_{X'}$ is nef, then $\kappa (X) = 0$,
but it is very much open whether $K_{X'} \equiv 0.$ Once this is known, it follows that $X$ 
is bimeromorphic to a quotient of a torus by a finite group.

The following proposition will be used to establish projectivity in section 3.

\begin{proposition} \label{ampleopen} 
Let $X$ be a compact K\"ahler manifold and $(Z_s)_{s\in S}$ a covering family of subvarieties. Assume that the general $Z_s$ is irreducible and reduced and
that some irreducible reduced member $Z_0$ has ample normal sheaf. Then the general member $Z_s$ is Moishezon. 
\end{proposition} 

\begin{proof} Let $q: U \to S$ be the graph of the family, with projection $p: U \to X.$   
We obtain an inclusion $p^*(\Omega^1_X) \to \Omega^1_U $ and in combination with the canonical surjection
$\Omega^1_U \to \Omega^1_{U/S} $ a map
$$ \alpha: p^*(\Omega^1_X) \to \Omega^1_{U/S}. $$
Let $\sS = {\rm Ker}(\alpha),$ a torsion free sheaf of rank say $r$. Take $s \in S$ such that $Z_s$ is irreducible and reduced and 
consider the complex-analytic fiber $\tilde Z_s := q^{-1}(s).$ Then $\tilde Z_s$ is generically reduced and set-theoretically we have $\tilde Z_s = 
Z_s,$ i.e., $Z_s$ is the reduction of $\tilde Z_s.$ It follows immediately that 
$$ (\sS \vert Z_s) / {\rm tor} = p^*(\sN^*_{Z_s/X}) / {\rm tor}. $$ 
We basically need this for $s = 0.$ Namely, let $\sT =  (\bigwedge^r \sS)^*.$ Then $\sT$ is a torsion free sheaf of rank $1$, and by our assumption
$(\sT \vert Z_0)/{\rm tor} $ is ample. We now first take normalizations $\tilde  U \to U$ and $\tilde S \to S$ and then a
desingularization $\hat U \to \tilde  U, $ inducing a  projection $\hat q: \hat U \to \tilde S.$ Let $s_0 \in \tilde S$ be a point over $0$
and $\hat Z_{s_0} $ be the set-theoretic fiber over $s_0$, which might be reducible. Let $A_0$ be the irreducible component of $\hat Z_{s_0}$ mapping onto $Z_{s_0}.$ 
Thus we have a birational map $A_0 \to Z_0.$ 
We may choose $\pi$ such that 
$$\pi^*(\sT) / {\rm tor} =: \hat \sT$$
is locally free. Then $\hat \sT \vert A_0$ is big and nef, since $\sT \vert Z_0$ is ample. Since $\hat U \to U$ is a projective morphism, we find a
line bundle $\sM$ on $\hat U$ such that $L = (\hat \sT)^{\otimes N} \otimes \sM $ is ample on every component of $\hat Z_{s_0}. $ Hence 
$L \vert \hat Z_{s_0}$ is ample, and therefore $L \vert \hat Z_s$ is ample for general $s$. Consequently the general $Z_s$ is Moishezon. 

\end{proof} 

\section{The algebraic dimension}
\setcounter{lemma}{0}

In this section we study the following 

\begin{conjecture} Let $X$ be a compact K\"ahler manifold and $Z \subset X$ an irreducible compact subvariety of dimension $d.$
Assume that the normal sheaf $\sN_{Z/X}$ is ample. Then $a(X) \geq d+1. $
\end{conjecture} 

In case $Z$ is a divisor, the line bundle $\sO_X(Y)$ is big (and nef), thus $X$ is Moishezon, hence projective. In higher codimension,
there are basically two results confirming the conjecture. 

\begin{theorem} [OP04] \label{OP} Let $X$ be a smooth compact K\"ahler threefold, $C \subset X$ an irreducible curve with ample normal sheaf. 
Then $a(X) \geq 2$ unless possibly $X$ is a simple threefold which is not bimeromorphic to a quotient of a torus by a finite group.
\end{theorem} 

\begin{theorem} [BM04] \label{BM} Let $X$ be a compact K\"ahler manifold of dimension $n,$ and $Y \subset X$ a locally complete intersection of dimension $p$ with ample
normal bundle. Assume furthermore that there is covering family $(Z_s)_{s \in S}$ of $q-$cycles with $p+q+1 = n.$ 
Then either $a(X) \geq p+1$ or the following holds. The compact irreducible parameter space $S$ is simple with $\dim S = p+1, $ the set 
$$ \Sigma = \{s \in S \ \vert \ Z_s \cap Y \ne \emptyset \ \} $$
has pure codimension 1 in $S$, and $S \setminus \Sigma $ is strongly pseudo-convex. 
\end{theorem} 

Specifying to $ p = 1,$ we obtain from Theorem~\ref{BM} 

\begin{corollary} \label{p1} Let $X$ be a compact K\"ahler manifold of dimension $n$ and $Y \subset X$ a smooth curve (or $1-$dimensional local complete intersection) with ample normal bundle. 
Suppose that $X$ is covered by subvarieties of codimension $2$. Then $a(X) \geq 2.$ 
\end{corollary} 

\begin{proof} By our assumption there is a covering family $(Z_s)_{s \in S}$ of $(n-2)
-$cycles. 
We apply the theorem of Barlet-Magnusson and need to exclude the second alternative in Theorem~\ref{BM}. So assume that $\dim S = 2, $ that $S$ is simple,
the set $\Sigma $ has dimension $1$ with strongly pseudo-convex complement $S \setminus \Sigma$. Now normalize $S$ and apply the following lemma to produce a
contradiction. 
\end{proof}

\begin{lemma} Let $S$ be a normal compact surface whose desingularisation is K\"ahler. Assume that there is an effective curve $\Sigma \subset S$
such that $S \setminus \Sigma $ is strongly pseudo-convex. Then $a(S) \geq 1.$ 
\end{lemma} 

\begin{proof} Since $S$ has only finitely many singularities, we may blow up and assume from the beginning that $S$ is smooth. 
Let $\tau: S \to S_0$ be a minimal model. Then, arguing by contradiction, $S$ is either a torus or a K3-surface with $a(S_0) = 0.$ In the K3 case, 
we contract all $(-2)-$curves and call the result again $S_0.$ Thus in both cases, $S_0$ is a (normal) surface with any curves. 
By our assumption, we find a non-constant holomorphic function $f \in \sO(S \setminus \Sigma).$ This function yields a non-constant holomorphic function on $S_0$
outside a finite set, which is absurd. 
\end{proof} 

\begin{remark} {\rm Corollary~\ref{p1} should be true for all $p.$ For this we would need to prove the following. 
\vskip .2cm \noindent
{\it Let $X$ be a normal compact K\"ahler space, $\Sigma \subset  X$ purely $1-$codimensional such that $X \setminus \Sigma$ 
strongly pseudo-convex. Then $X$ cannot be simple.} 
\vskip .2cm \noindent
Assume $X$ is simple and $\dim X = 3$. As explained in Section 2, $X$ should be bimeromorphic to $T/G$, 
where $T$ is a simple torus and $G$ a finite group. We verify the above assertion in this case. So let $\Sigma \subset X$ be purely 1-codimensional such that 
$X \setminus \Sigma$ is strongly pseudo-convex. 
Let $\pi: \hat X \to X$ be bimeromorphic such that $\hat X$ admits a holomorphic bimeromorphic map $f: \hat X \to T/G.$ Let $\hat \Sigma $
be the preimage of $\Sigma. $ Then $\hat X \setminus \hat \Sigma $ carries non-constant holomorphic functions. On the other hand,
$\dim f(\hat \Sigma) = 0.$ This is a contradiction. \\
Of course, this argument holds in all dimensions.} 
\end{remark}

In the following we adress the question whether in Theorem~\ref{OP} the case $a(X) = 2$ can really occur, following ideas in [OP04].

\begin{ps-example} {\rm Let $X$ be a smooth compact K\"ahler threefold with $a(X) = 2.$ \\
We assume that we have a holomorphic algebraic reduction  
$$f: X \to S $$ to a smooth projective surface $S$
with the following properties:
\begin{enumerate}
\item there is an irreducible curve $B \subset S$ with $B^2 > 0$ whose preimage $X_B = f^{-1}(B)$ is irreducible;
\item the general fiber of $f \vert X_B $ is a singular rational curve (with a simple cusp or node);
\item $X_B$ is projective.
\end{enumerate} 
\vskip .2cm 
Notice that $X_B$ is always Moishezon, but the projectivity is not automatic. 
\vskip .2cm
Given these data, we choose a general hyperplane section $C \subset X_B$, hence $C$ is a local complete intersection in $X.$ 
Furthermore we have an exact sequence of vector bundles
$$ 0 \to N_{C/X_B} \to N_{C/X} \to N_{X_B/X} \vert C \to 0. $$
Since $N_{X_B/X} \vert C = f^*(N_{B/S}) \vert C$ is ample, the bundle $N_{C/X}$ is ample, too. \\
If $B$ is smooth, so does $C$. 
\vskip .2cm 
Certainly a K\"ahler threefold $X$ with the first two conditions must exist, although an explicit construction seems not so easy. 
Also it is plausible that the third condition should hold in certain cases.  \\
\vskip .2cm 
It is easy to fulfill all conditions if one allows $B^2 = 0.$ Here are the details. 
We start with a K\"ahler surface $S_1$ of algebraic dimension 1 with algebraic reduction $f_1: S_1 \to T = \bP_1.$ We may choose $S_1$
such that there is a point $x_0 \in T$ so that the fiber $f_1^{-1}(x_0)$ is an irreducible rational curve with a simple cusp or node. 
Let $S_2 $ be $\bP_2$ blown up in 9 points so that there is an elliptic fibration $f_2: S_2 \to T$.
Set  
$$ X = S_1 \times_T S_2. $$ 
Here we have arranged things so that $X$ is smooth, by arranging $f_2$ to be smooth over the singular set of $f_1.$  
The projection $h: X \to S_2$ is the algebraic reduction; in particular $a(X) = 2.$ 
Let $B = f^{-1}_2(x_0) $, an elliptic curve, and 
observe that $X_B  \simeq f_1^{-1}(x_0) \times B$ which is projective. }
\end{ps-example} 

If $Z$ is a subvariety with ample normal sheaf moving in a covering family, things get much easier. 

\begin{theorem} Let $X$ be a compact K\"ahler manifold and $Z \subset  X$ an irreducible reduced subspace with ample normal sheaf.
Assume that $Z$ moves in a generically irreducible and reduced family $(Z_s)_{s \in S}$ which covers $X.$ 
Then $X$ is projective.
\end{theorem} 

\begin{proof} By Proposition \ref{ampleopen} the general $Z_s$ is Moishezon. 
We may assume that the family $(Z_s)$ is not connecting, i.e. two general points cannot be connected by a chain of curves $Z_s,$
otherwise $X$ is already projective by Campana \cite{Ca81}, since then
$X$ is {\it algebraically} connected. 
Hence we may consider the quotient of the family, yielding an almost holomorphic map $f: X \dasharrow W$, which contracts
two general points to the same point in $W$ iff they can be joined by a chain of members $Z_s$; cp. \cite{Ca81}, \cite{Ca04}. 
Since the family is not connecting, we have $\dim W > 0.$  Now the general $Z_s$ is contained in a compact fiber $X_w.$ Thus 
we get a generically surjective map
$$ \sN_{Z_s/X} \to \sN_{X_w/X} \vert Z_s \simeq \sO_{Z_s}^{\oplus k}. $$
This contradicts the ampleness of $ \sN_{Z_s/X}.$
\end{proof}

Notice that the ampleness of the normal sheaf of $Z$ does not necessarily means that even a multiple of $Z$ moves; see \cite{FL82}
for an example. \\
We adress finally the interesting case that $X$ is a hyperk\"ahler manifold. 

\begin{theorem} \label{HK} Let $X$ be a compact hyperk\"ahler manifold, $C \subset X$ an irreducible curve with ample normal sheaf. 
If $a(X) \geq 1,$ then $X$ is projective.
\end{theorem} 

\begin{proof} Suppose $X$ not projective. 
Following some arguments in [COP10] (3.4), let $g: X \dasharrow B$ be an algebraic reduction, $\pi: \hat X \to X$ a bimeromorpic
map from a compact K\"ahler manifold $\hat X$ such that the induced map $f:\hat X \to B$ is holomorphic. 
Fix an ample line bundle $A$ on $B$ and set
$$ \sL = (\pi_* f^*(A))^{**}. $$
Then by [COP10] (3.4), the line bundle $\sL$ is nef with $\sL \cdot C = 0$. 
Since $\sL$ is effective, this contradicts Theorem \ref{griffiths} resp. Corollary~\ref{cor1} below. 
\end{proof} 

\begin{remark} {\rm Theorem~\ref{HK} should hold without the assumption that $a(X) \geq 1.$ In other words, a hyperk\"ahler manifold $X$ of
dimension $2n$ 
containing an irreducible curve $C$ with ample normal sheaf should be projective. Let $q_X$ be the Beauville form. Then we have an 
isomorphism
$$ \iota: H^{1,1}(X,\bQ) \to H^{2n-1,2n-1}(X,\bQ), $$
see [COP10], p.411, for details. In particular, there exists $u \in H^{1,1}(X,\bQ)$ such $\iota(u) = [C], $ which is to say that 
$$ a \cdot C = q_X(a,u) $$
for all $a \in H^{1,1}(X,\bQ).$ 
Since $u$ is a rational class, there exists a positive rational number $\lambda$ and a line bundle $L$ such that $u = \lambda c_1(L).$ 
The hope now is that the positivity of the normal sheaf of $C$ implies that $L$ is nef (and that $L$ is semi-ample). }  
\end{remark}

If $C \subset X$ is a smooth curve with small genus and ample normal bundle, we have the following algebraicity result \cite{OP04}:

\begin{proposition} Let $X$ be a compact K\"ahler manifold and $C \subset X$ a smooth curve with ample normal bundle. 
If $g(C) \leq 1,$ the manifold $X$ is projective.
\end{proposition} 

Actually much more holds. If $g(C) = 0,$ then $X$ is rationally connected, and if $g(C) = 1,$ then either $X$ is rationally connected or
the rational quotient has $1-$dimensional image. In the latter case we have a holomorphic map $f: X \to W$ with rationally connected fiber
to an elliptic curve $W$ and $B$ is an \'etale multi-section. 

This can be generalised to higher dimensions in the following way.

\begin{theorem} \label{kappa} Let $X$ be a compact K\"ahler manifold and $Z \subset  X$ a compact submanifold with ample normal bundle. If 
$Z$ is uniruled, $X$ is projective. 
\end{theorem} 

\begin{proof} 
Since $Z$ is uniruled and the normal bundle $N_{Z/X}$ is ample, it is clear that $X$ is uniruled, too. Let $f: X \dasharrow W$ be 
``the'' rational quotient. Then we apply Theorem 3.7 in \cite{Pe06} to conclude that $Z$ dominates $W$ (we may apply Theorem 3.7, since
the essential ingredient Lemma 3.6 works also in the K\"ahler case). Since $Y$ is projective as well as the fibers of $f,$ we conclude 
that any two points of $X$ can be joined by a chain of compact curves, hence $X$ is projective by \cite{Ca81}. 
\end{proof}

\begin{remark} {\rm Theorem~\ref{kappa} should remain true if we assume $\kappa (Z) < \dim Y$ instead of the uniruledness of $Z$. 
In fact, the K\"ahler version of Theorem 3.2 in \cite{Pe06} proves that at least $K_X$ is not pseudo-effective. It is conjectured, but completely
open in dimension at least $4$, that then $X$ is uniruled. Once we know the uniruledness, we conclude as before. \\
One can also prove versions of the preceeding theorem weakening the ampleness condition.} 
\end{remark}

\section{Curves with ample normal bundles and the cone of curves} 
\setcounter{lemma}{0}

If $X$ is a projective manifold containing a hypersurface $Y$ with ample normal bundle, then - as already mentioned - 
the line bundle $\sO_X(Y)$ is big and therefore
the class $[Y]$ is in the interior of the effective cone of $X$. Dually we expect
 
\begin{conjecture} 
Let $X$ be a projective manifold, $C \subset X$ an irreducible curve. If the normal sheaf $N_{C/X}$ is ample,
then $[C]$ is in the interior of the Mori cone $\overline{NE}(X).$ 
\end{conjecture}

This conjecture can be restated as follows. 

\begin{conjecture} \label{conjA} Let $X$ be a projective manifold, $L$ a nef line bundle and $C \subset X$ an irreducible curve with
ample normal bundle. If $L \cdot C = 0, $ then $L \equiv 0.$
\end{conjecture} 

It is interesting to note that in codimension different from $1$ and $n-1$, the corresponding statement is false, as demonstrated by an 
example of Voisin [Vo08]. \\
We consider first the case when $L$ has a section with smooth zero locus and when $C$ is smooth:

\begin{theorem} \label{griffiths} 
Let $X$ be a projective manifold, $Y \subset X$ a smooth hypersurface with nef normal bundle. Let $C \subset X$ 
be a smooth curve with ample normal bundle. Then $Y \cap C \ne \emptyset.$ 
\end{theorem} 

\begin{proof} Let $n = \dim X.$ Suppose to the contrary that $Y \cap C = \emptyset$ and set $Z = X \setminus C. $ 
By [Um73], the normal bundle $N_{C/X} $ is Griffiths-positive, hence by Schneider [Sch73], $Z$ is $(n-1)-$convex in the sense of Andreotti-Grauert [AG62]. By the finiteness theorem of Andreotti-Grauert 
$$ \dim H^{n-1}(Z,\sF)$$
for any coherent sheaf $\sF$ on $Z.$ 
We will use in the following only line bundles $\sF.$ 
Let $Y_k$ be the $k-$th infinitesimal neighborhood $Y,$ i.e., $Y$ is defined by the ideal $\sI_Y^k.$ 
Consider the exact sequence 
$$ H^{n-1}(Z,\sF) \to H^{n-1}(Y_k,\sF) \to H^n(Z,\sI_Y^k \otimes \sF). $$
The last group vanishing due to the non-compactness of $Z$ (Siu [Si69]), we conclude that 
$\dim H^{n-1}(Y_k,\sF)$ is bounded from above: there is a constant $M > 0$ such that
$$ \dim H^{n-1}(Y_k,\sF) \leq M  \leqno (1) $$
for all positive $k.$ 
\vskip ,2cm \noindent 
Now choose $\sF$ to be a negative line bundle on $Y.$ Then by Kodaira vanishing
$$ H^1(Y,K_Y \otimes N_Y^{\mu} \otimes \sF^*) = 0\leqno (2) $$
for all $\mu \geq 0.$ 
Dually $$H^{n-2}(Y,N^{*\mu}_Y \otimes \sF) = 0. $$
Therefore we obtain an exact sequence
$$ 0 \to H^{n-2}(Y_k,\sF) \to H^{n-2}(Y_{k-1},\sF) \to H^{n-1}(Y,(N_Y^*)^k \otimes \sF) \to $$
$$\to H^{n-1}(Y_k,\sF) 
{\buildrel {a_k} \over {\to}} H^{n-1}(Y_{k-1},\sF) \to 0.$$
By the boundedness statement (1), $a_k$ is an isomorphism for $k \gg 0.$ 
Thus 
$$  H^{n-1}(Y,(N_Y^*)^k \otimes \sF) = 0 $$
for $k \gg 0.$
Dualizing
$$ H^0(Y,K_Y \otimes N_Y^k \otimes \sF^*) = 0$$
for all $k \geq k_0(\sF)).$ 
Setting for simplicity 
$B = N_Y $ and $A = \sF^*$, we are in the following situation: \\
{\it $B$ is a nef line bundle on $Y$ such that for all ample line bundles $A$ there is a number $k_0$ such that 
$$ H^0(Y,K_Y \otimes kB \otimes A) = 0$$
for $k \geq k_0(A).$ } \\
Equivalently by Kodaira vanishing
$$ \chi(Y,K_Y \otimes kB \otimes A) = 0. $$
This is clearly impossible by Riemann-Roch, reaching a contradiction. 
Actually we do need to consider all ample $A$ here; it suffices to take for $A$ the powers of a
fixed ample line bundle.

\end{proof}

\begin{corollary} \label{cor1} Let $L = \sO_X(Y)$ be a nef line bundle with $Y$ smooth. Let $C \subset X$ be a smooth curve with
ample normal bundle. Then $L \cdot C > 0. $
\end{corollary} 

\begin{proof} Assume $L \cdot C = 0.$ Then by Theorem \ref{griffiths}, $C \subset Y.$ 
Now the normal bundle sequence
$$ 0 \to N_{C/Y} \to N_{C/X} \to N_{Y/X} \vert C = \sO_C \to 0 $$
contradicts the ampleness of $N_{C/X}. $ 
\end{proof}

\begin{theorem} \label{thmg2}  Theorem \ref{griffiths} remains true for singular, possibly non-reduced, reducible hypersurfaces $Y.$
\end{theorem}

\begin{proof} The proof in the smooth case basically goes over, with the following modifications. 
Of course, the use of Kodaira vanishing (2) is critical. 
We fix a negative line bundle $\sF$ which we may choose as restriction of a negative line bundle $\tilde \sF$ on $X.$ 
Then we can apply Kodaira vanishing on $X$ to obtain the vanishing (2), also for higher $H^q$'s. 
Namely 
$$ H^q(X,K_X \otimes \tilde \sF^* \otimes \sO_X((\mu + 1)Y)) = H^{q+1}(X,K_X \otimes \tilde \sF^* \otimes \sO_X(\mu Y)) = $$
implies (using the adjunction formula)
$$ H^q(Y,K_Y \otimes \sF^* \otimes N_Y^{\mu}) = 0. $$
At the end we compute $\chi(K_Y \otimes kB \otimes A) $ via Riemann-Roch on $X.$ 
In fact, we obtain as before that
$$ \chi(Y,K_Y \otimes N_Y^k \otimes A) = 0 \eqno (*) $$
for all extendible ample line bundle $A$ on $Y$ and $k \geq k_0(A).$ Here extendibility means that there is an ample
line bundle $\tilde A$ on $X$ such that $\tilde A \vert Y =  A.$ 
Then by (*) 
$ \chi(X, K_X \otimes \sO(kY) \otimes \tilde A)  $
is constant
for large (hence all) $k.$ By Riemann-Roch this immediately implies $Y \equiv 0,$ which is absurd. 
\end{proof} 

\begin{corollary} Conjecture \ref{conjA} holds, if $L$ is effective and $C$ smooth (with ample normal bundle). 
\end{corollary} 

This is a consequence of Theorem \ref{thmg2} and the following 

\begin{lemma} Let  $Y = \sum_{i=1}^N m_iY_i$ be an effective nef divisor on the projective manifold $X.$ Let $C \subset X$ 
be an irreducible curve with ample normal sheaf. Suppose that ${\rm supp} Y \cap C \ne \emptyset.$ Then $Y \cdot C \ne 0.$
\end{lemma} 

\begin{proof} We may assume that $C \subset {\rm supp}(Y).$ 
Let $Y_1, \ldots, Y_s$ be the components $Y_j$ such that $C \subset Y_j.$ 
We may furthermore assume that $Y_j \cdot C \leq 0$ for some $j$, 
otherwise we are already done. After renumbering, we have $j = 1.$  
Consider the canonical map
$$ \kappa: N^*_{Y_1/X} \vert C \to N^*_{C/X}. $$
Since $N^*_{Y_1/X} \vert C$ is nef, $\kappa \vert C = 0.$ This is to say that $C \subset {\rm Sing}(Y_1). $ 
By taking power series expansions of the local equation of $Y_1$, we obtain a number $k \geq 2$ and a non-zero map
$$ N^*_{Y_1/X} \vert C \to S^kN^*_{C/X}, $$
contradicting again the nefness of  $N^*_{Y_1/X} \vert C$.
\end{proof}

The proof of Theorem \ref{thmg2} actually shows

\begin{corollary} Let $Z$ be an $(n-1)-$convex complex manifold of dimension $n$ and $Y \subset X$ a compact hypersurface. Then 
the normal bundle $N_{Y/Z}$ cannot be nef. 
\end{corollary} 

This leads to the following

\begin{question} \label{qu} {\rm Let $X$ be a $q-$convex manifold and $Y \subset X$ a compact subvariety with nef normal sheaf. 
Is then $\dim Y \leq q-1?$ } 
\end{question} 

Besides the case $q = n-1,$ this question has a positive answer also for $q = 1,$ because then 
there exists a proper modification $\phi: X \to Y$ to a Stein space $Y$. Thus $\dim \phi(C) = 0$ which 
easily contradicts the nefness of the normal sheaf of the curve $Y$. \\
Even if $Y$ has ample normal sheaf, Question~\ref{qu} is wide open. In fact a positive answer to Question~\ref{qu} would imply a solution 
to the following conjecture of Hartshorne 
\vskip .2cm 
{\it Let $Z$ be a projective manifold contaning submanifolds $X$ and $Y$ with ample normal bundles. 
If $\dim X + \dim Y \geq  \dim Z,$ then $X \cap Y \ne \emptyset.$ } 

\vskip .2cm \noindent
For further informations on this conjecture, we refer to \cite{Pe09}. 
\vskip .2cm
The connection to Question~\ref{qu} is provided by the convexity of $Z \setminus X$ resp. of $Z \setminus Y$. 

\begin{remark} {\rm  If $C$ is a {\it singular} curve, Theorem \ref{thmg2} should essentially remain valid.   
The only point which needs to be shown is the $(n-1)-$convexity of the complement $X \setminus C,$ which is somehow subtle. 
Here is what is known. If $C$ is locally a complete intersection, there is no problem, [Sch73], things basically work as in the smooth case.
In general, it follows from the results of Fritzsche [Fr76],[Fr77], that $X \setminus C$ is $(n-1)-$convex, provided
that the rank of the conormal sheaf $\sI/\sI^2$ at every point is at most $n-1$, which means of course that $\sI/\sI^2$ is  locally free of rank $n-1.$} 
\end{remark} 

Instead making assumptions on the line bundle $L$, one might instead impose conditions on $C$.

\begin{theorem} Let $X$ be a projective manifold, $C \subset X$ be an irreducible curve with ample
normal bundle. Assume that $C$ moves in a family $(C_s)$ covering $X$. Let $L$ be a nef line bundle with $L \cdot C = 0.$
Then $L \equiv 0.$ 
\end{theorem} 

\begin{proof} Let $f: X \dasharrow W$ be the nef reduction of $L$, cite{workshop}  The map $f$ is almost holomorphic and, 
due to the existence
of the family $(C_s),$ the map $f$ is not trivial: $\dim W < \dim X.$ 
A general member $C_s$ has still ample normal bundle; on the other hand $C_s$ is contained in a (compact) fiber of $f.$ 
This is only possible when $\dim W = 0.$ Hence $L \equiv 0 $ by [workshop]. 
\end{proof} 

It is of course not true that a curve with ample normal bundle is in the interior of the movable cone $\overline{ME}(X). $
Simply start with a projective manifold $Y$ containing a curve $C$ with ample normal bundle and let $\pi: X \to Y$
be the blow-up of $Y$ at a point $Y \not \in C.$ If $E \subset  X$ is the exceptional divisor, then $E \cdot C = 0$,
hence $[C]$ is in the boundary of $\overline{ME}(X),$ since $\overline{ME}(X)$ is the dual cone of the pseudo-effective
cone by [BDPP04]. 
In some vague sense this however should be the only obstruction:

\begin{conjecture} \label{conjB} Let $X$ be a projective manifold and $C \subset X$ a curve with ample normal sheaf. 
Let $L$ be a pseudo-effective line bundle with $L \cdot C = 0.$ Then the numerical dimension $\nu(L) = 0.$
\end{conjecture} 

For the notion of the numerical dimension of a pseudo-effective line bundle we refer to [Bou04] and [BDPP04]. 
In case $L$ is nef, $\nu(L) = 0$ just says that $L \equiv 0 $, so that Conjecture \ref{conjB} implies \ref{conjA}
(which is of course clear from the point of view of cones: $\overline{ME}(X) \subset \overline{NE}(X)). $ 

Here is some evidence for Conjecture \ref{conjB}.

\begin{proposition} Let $L$ be a line bundle and $C \subset X$ an irreducible curve with ample
normal sheaf. Assume that $L \cdot C = 0.$ Then $\kappa (L) \leq 0.$ 
\end{proposition} 

\begin{proof}
This is a direct consequence of [PSS98], Theorem 2.1. 
\end{proof} 

\begin{remark} {\rm A pseudo-effective line bundle $L$ admits by Boucksom \cite{Bou04} a so-called divisorial Zariski decomposition:
$$ L \equiv M + E, $$
where $M$ is an $\bR-$divisor which is nef in codimension 1 and $E$ is an effective $\bR-$divisor. 
Suppose that $L \cdot C = $ for a curve $C$ with ample normal sheaf. Since $[C] \in \overline{ME}(X),$ we have
$E \cdot C \geq 0$ and $M \cdot C \geq 0,$
hence   $$ M \cdot C  = 0.$$ 
Notice that $M \equiv 0$ is equivalent to $\nu(L) = 0.$ 
Therefore Conjecture~\ref{conjB} is equivalent to the following conjecture}
\end{remark} 

\begin{conjecture} Let $X$ be a projective manifold, $C \subset  X$ an irreducible curve with ample normal sheaf and 
$L$ an $\bR-$divisor which is nef in codimension 1. If $L \cdot C = 0,$ then $L \equiv 0.$ 
\end{conjecture}

In order to show that the class $[C]$ of an irreducible curve $C \subset X$ is in the interior of the movable cone
one needs a more global assumption than just the ampleness of the normal sheaf. We shall use the following 
notation introduced in \cite{PSS99}.

\begin{definition} \label{defflag}
Let $X$ be a projective manifold. A sequence 
$$ Y_q \subset Y_{q+1} \subset \ldots \subset  Y_n = X $$
of $k-$dimensional irreducible subvarieties $Y_k \subset X$ is an ample $q-$flag if for every $q  \leq k \leq n-1$ there is an
ample Cartier divisor $D_k$ on the normalization $\eta_{k+1}: \tilde Y_{k+1} \to Y_{k+1} $ such that $Y_k = \eta_{k+1}({\rm supp}(D_{k+1})). $
\end{definition}

\begin{remark} {\rm The main result in \cite{BDPP04} implies that the closed cone generated by the classes of curves appearing as the first 
member of an ample $(n-1)-$flag is the movable cone. }
\end{remark} 

Now we have easily

\begin{theorem}  Let $X$ be a projective manifold, $C \subset X$ an irreducible curve appearing in ample flag
$$ C = Y_1 \subset \ldots \subset  Y_q \ldots \subset  Y_n = X. $$
Then $[C]$ is in the interior of $\overline{ME}(X).$ 
\end{theorem} 

\begin{proof} We need to prove the following statement. 
\vskip .2cm 
{\it Let $L$ be a pseudo-effective line bundle with $L \cdot C = 0. $ Then $L \equiv 0.$ }
\vskip .2cm \noindent
We prove inductively that $L \vert Y_i \equiv 0$ for all $i.$ \\
The claim for $i = 1$ is our assumption $L \cdot C = 0.$ So suppose the statement for $i$, i.e. 
$L \vert Y_i \equiv 0.$ Using the notations of Definition~\ref{defflag} we find an ample divisor $D_{i+1} $ on
$\tilde Y_{i+1}$ such that 
$$ \eta_{i+1}^*(L)  \vert\  {\rm supp} D_{i+1} \equiv 0 $$
Now we conclude by the following proposition. 
\end{proof}

\begin{proposition} Let $X$ be a normal projective variety, $D = \sum m_i D_i$ an ample divisor and $L$ a pseudo-effective
line bundle. If $L \vert D_i \equiv 0$ for all $i,$ then $L \equiv 0.$ 
\end{proposition} 

\begin{proof} By assumption, $L \cdot D_i = 0$ for all $i,$ hence $L \cdot D = 0.$ 
Let $H_1, \ldots, H_{n-2}$ be arbitrary very ample divisors. Then 
$$L \cdot H_1 \cdot \ldots \cdot H_{n-2} \in \overline{NE}(X),$$
since $L$ is pseudo-effective. Hence the vanishing 
$$L \cdot H_1 \cdot \ldots \cdot H_{n-2} \cdot D = 0$$
together with the ampleness of $D$ implies
$$L \cdot H_1 \cdot \ldots \cdot H_{n-2} = 0.$$ 
This is only possible when $L \equiv 0 .$
\end{proof}

\vskip 2cm \noindent
Thomas Peternell \\
Mathematisches Institut, Universit\" at Bayreuth \\
D-95440 Bayreuth, Germany \\
thomas.peternell@uni-bayreuth.de

\end{document}